\documentclass[a4paper,12pt]{article}

\usepackage[utf8]{inputenc}
\usepackage[english]{babel}
\usepackage{enumitem}
\usepackage{float}
\usepackage{amsmath}
\usepackage{amssymb}
\usepackage{amsthm}
\usepackage{amsfonts}
\usepackage{booktabs}
\usepackage{epsfig}
\usepackage{epstopdf}
\usepackage{cite}
\usepackage{multirow} 
\usepackage{algorithm,algpseudocode}
\usepackage{subcaption}
\usepackage{xcolor}
\usepackage[linkcolor=blue,colorlinks=true]{hyperref}
\let\originaleqref\eqref
\renewcommand{\eqref}{~\originaleqref}

\usepackage{caption}
\newtheorem{dfn}{Definition}[section]
\newtheorem{lem}[dfn]{Lemma}
\newtheorem{thm}[dfn]{Theorem}

\newtheorem{exm}[dfn]{Example}
\theoremstyle{definition}

\newtheorem{rem}[dfn]{Remark}

\allowdisplaybreaks

\title{Convergence Results of Two-Step Inertial Proximal Point Algorithm }

\date{\today}

\author{Olaniyi S. Iyiola\footnote{Department of Mathematics, Clarkson University, Potsdam, NY, USA; e-mail:niyi4oau@gmail.com; oiyiola@clarkson.edu}
\hspace*{0.8mm}, Yekini Shehu\footnote{College of Mathematics and Computer Science, Zhejiang Normal University, Jinhua 321004, People’s Republic of China; e-mail: yekini.shehu@zjnu.edu.cn}}

\begin{document}

\maketitle

\begin{abstract}
\noindent
This paper proposes a two-point inertial proximal point algorithm to find zero of maximal monotone
operators in Hilbert spaces. We obtain weak convergence results and non-asymptotic $O(1/n)$ convergence rate of our proposed algorithm in non-ergodic sense. Applications of our results to various well-known convex optimization methods, such as the proximal method of multipliers and the alternating direction method of multipliers are given. Numerical results are given to demonstrate the accelerating behaviors of our method over other related methods in the literature.\\



\end{abstract}

\section{Introduction}\label{Sec:Intro}
\noindent Suppose $H$ is a real Hilbert space with inner product $\langle.,.\rangle$ and induced norm $\|\cdot\|.$  Given a maximal monotone set-valued operator, $A:H \to 2^{H},$  we consider the following inclusion problem:
\begin{equation}\label{ocn}
\mbox{find  } x   \in H \ \mbox{ such that } \ \textbf{0} \in A(x).
 \end{equation}
\noindent Throughout this paper, we shall denote by $A^{-1}(\textbf{0})$, the set of solutions to \eqref{ocn}. It is well known that the inclusion problem \eqref{ocn} serves as a unifying  model for many problems of fundamental importance, including
fixed point problem,  variational inequality problem,  minimization of closed proper convex functions, and
their variants and extensions. Therefore, its efficient solution is of practical interest  in many situations.\\

\noindent
\textbf{Related works.}  The proximal point algorithm (PPA), which was first  studied by Martinet and further developed by Rockafellar and others (see, for example, \cite{8a,26M,25,26,34}) has been used for many years for studying the inclusion problem \eqref{ocn}. Starting from an arbitrary point  $x_0 \in H,$  the PPA iteratively generates its sequence $\{x_n\}$  by
\begin{equation}\label{oct}
x_{n+1}=J_{\lambda}^A(x_n),
 \end{equation}
(where the operator
 $J_{\lambda}^A:=(I+\lambda A)^{-1}$ is the so-called resolvent operator, that has been
introduced by Moreau in \cite{25}) which is equivalent to
\begin{equation}\label{octa}
\textbf{0} \in \lambda A(x_{n+1})+x_{n+1}-x_n
\end{equation}
where $\lambda >0$  called proximal parameter. The PPA \eqref{oct} is a very powerful algorithmic tool and contains many well known algorithms as special cases, including the classical augmented Lagrangian method  \cite{20,29}, the Douglas-Rachford splitting method \cite{DouglasRachford,LionsMercier} and the alternating direction method of multipliers \cite{17,18}.  Interesting results on weak convergence and rate of convergence of PPA have been obtained in \cite{12,19,Min}. The equivalent representation of the PPA \eqref{octa}, can be written as
\begin{equation}\label{oct1}
\textbf{0} \in \frac{ x_{n+1}-x_n}{\lambda}+ A(x_{n+1}).
\end{equation}
This  can be viewed as an implicit discretization of the evolution differential inclusion problem
\begin{equation}\label{octa2}
\textbf{0} \in \frac{dx}{dt} + A(x(t))
\end{equation}
\noindent It has been shown that  the solution trajectory of \eqref{octa2} converges to a solution of \eqref{ocn}  provided that $T$  satisfies certain conditions  (see, for example, \cite{10}).\\

\noindent
To speed up convergence of PPA \eqref{oct}, the following second order evolution differential inclusion
problem was introduced  in the literature:
\begin{equation}\label{octa3}
\textbf{0} \in \frac{d^{2}x}{dt^2} + \beta\frac{dx}{dt} + A(x(t)),
\end{equation}
\noindent  where $\beta > 0$  is a friction parameter.  If $A = \nabla f,$  where $ f: \mathbb{R}^2 \to \mathbb{R}$  is a differentiable
convex function with attainable minimum, the system \eqref{octa3} characterizes roughly the motion of a heavy ball which
rolls under its own inertia over the graph of $f$ until friction stops it at a stationary point of $f.$  In this case, the three
terms in  \eqref{octa3} denote, respectively, inertial force, friction force and gravity force. Consequently, the system \eqref{octa3} is usually
referred to as the heavy-ball with friction (HBF) system (see \cite{Pol}). In theory,  the convergence of the solution trajectories of the
HBF system to a solution of  \eqref{ocn} can be faster than those of the first-order system  \eqref{octa2}, while in practice the second
order inertial term  $\frac{d^2 x}{dt^2}$   can be exploited to design faster algorithms (see, e.g., \cite{1,5}).
As a result of  the properties of \eqref{octa3}, an implicit discretization method was proposed in \cite{2,4} as follows, given $x_{n-1}$  and $x_n$, the next point  $x_{n+1}$  is
determined via
\begin{equation}\label{adedollar}
\textbf{0}  \in \frac{x_{n+1}-2x_n+x_{n-1}}{h^2} + \beta\frac{x_{n+1}-x_n}{h}+ A(x_{n+1}),
\end{equation}
which result to an iterative algorithm of the form
\begin{equation}\label{octa4}
x_{n+1} = J_{\lambda}^A(x_n+\theta(x_n-x_{n-1})),
\end{equation}
where $\lambda= \frac{h^2}{1+ \beta h}$  and  $\theta= \frac{1}{1+ \beta h}.$
In fact, if we replace $ x_{n+1}-2x_n+x_{n-1}$ in \eqref{adedollar} with
$ x_{n+1}-x_n+\rho(x_{n-1}-x_n), \rho \in [0,1]$, we obtain a general algorithm of \eqref{octa4} with
  $\lambda= \frac{h^2}{1+ \beta h}$  and  $\theta= \frac{\rho}{1+ \beta h}, \rho\in [0,1].$
Note also that \eqref{octa4} is the  proximal point step applied to the
extrapolated point  $x_n+\theta(x_n-x_{n-1})$  rather than  $x_n$  as in the classical PPA \eqref{oct}.  We call the iterative method in  \eqref{octa4} one-step inertial PPA. Convergence properties of  \eqref{octa4} have been studied in  \cite{2,3,4,21,22,26} under some assumptions on the parameters $\theta$ and $\lambda.$ The inertial PPA \eqref{octa4} has been adapted to studying inertial Douglas-Rachford splitting method \cite{6,CA,9,C,DouglasRachford,LionsMercier}, inertial alternating method of multipliers (ADMM) \cite{B,7,C,18,19} and demonstrated their performance numerically on some imaging and data analysis problems. In all the references mentioned above, the inertial PPA \eqref{octa4} (which is the PPA \eqref{oct} with the one-step inertial extrapolation) has been studied. This is a different approach we that we take in this paper, where we consider the PPA \eqref{oct} with the two-step inertial extrapolation. Our result is motivated by the results given in \cite{Kim}, where an accelerated proximal point algorithm (which involves both the one-step inertial term and correction term) for maximal monotone operator is studied. In contrast to the method of Kim \cite{Kim}, we replace iterate $y_{n-1}$ in the correction term of \cite{Kim} with $x_{n-1}$ to obtain two-step inertial extrapolation and investigate the convergence properties.\\

\noindent
 From another point of view, our proposed two-step inertial PPA can be regarded as a general parametrized proximal point algorithm. Recent and interesting results on parametrized proximal point algorithm can be found in \cite{Bai1,MaNi}, where parametrized proximal point algorithm is developed for solving a class of separable convex programming problems subject to linear and convex constraints.
In these papers \cite{Bai1,MaNi}, it was shown numerically that parametrized proximal point algorithm could perform significantly better for solving sparse optimization problems than ADMM and relaxed proximal point algorithm.\\

\noindent
\textbf{Advantages of two-step proximal point algorithms.}
In \cite{Poon1,Poon2}, Poon and Liang discussed some limitations of inertial Douglas-Rachford splitting method and inertial ADMM. For example, consider the following feasibility problem in $\mathbb{R}^2$.

\begin{exm}\label{voice}
Let $T_1, T_2 \subset \mathbb{R}^2$ be two subspaces such that $T_1 \cap T_2 \neq \emptyset$. Find $x \in \mathbb{R}^2$ such that $x \in
T_1 \cap T_2$.
\end{exm}
\noindent
It was shown in \cite[Section 4]{Poon2} that two-step inertial Douglas-Rachford splitting method, where
$$x_{n+1}=F_{DR}(x_n+\theta(x_n-x_{n-1})+\delta(x_{n-1}-x_{n-2}))$$
converges faster than one-step inertial Douglas-Rachford splitting method
$$x_{n+1}=F_{DR}(x_n+\theta(x_n-x_{n-1}))$$
for Example \ref{voice}. In fact, it was shown using this Example \ref{voice} that one-step inertial Douglas-Rachford splitting method
$$x_{n+1}=F_{DR}(x_n+\theta(x_n-x_{n-1}))$$
converges slower than the Douglas-Rachford splitting method $$x_{n+1}=F_{DR}(x_n),$$
where
$$F_{DR}:=\frac{1}{2}\Big(I+(2P_{T_1}-I)(2P_{T_2}-I)\Big)$$
is the Douglas-Rachford splitting operator. This example therefore shows that one-step inertial Douglas-Rachford splitting method may fail to provide acceleration. Therefore, for certain cases, the use of inertia of more than two points could be beneficial. It was remark in \cite[Chapter 4]{Liang} that the use of more than two points $x_n, x_{n-1}$ could provide acceleration. For example,  the following two-step inertial extrapolation
\begin{equation}\label{afr}
y_n=x_n+\theta(x_n-x_{n-1})+\delta(x_{n-1}-x_{n-2})
\end{equation}
with $\theta > 0$ and $\delta< 0$ can provide acceleration. The failure of one-step inertial acceleration of ADMM was also discussed in \cite[Section 3]{Poon1} and adaptive acceleration for ADMM was proposed instead. Polyak \cite{Polyakbook} also discussed that the multi-step inertial methods can boost the speed of optimization methods even though neither the convergence nor the rate result of such multi-step inertial methods is established in \cite{Polyakbook}. Some results on multi-step inertial methods have recently been studied in \cite{CombettesGlaudin,DongJOGO}. \\

\noindent
\textbf{Our contribution.}  In this paper, we propose an inertial proximal point algorithm with two-step inertial extrapolation step. We obtain weak convergence results and give non-asymptotic $O(1/n)$ convergence rate of our proposed algorithm in non-ergodic sense. The summability conditions of the inertial parameters and the sequence of iterates imposed in \cite[Algorithm 1.2]{CombettesGlaudin}, \cite[Theorem 4.2 (35)]{DongJOGO}, and \cite[Chapter 4, (4.2.5)]{Liang} are dispensed with in our results. We apply our results  to the proximal method of multipliers and the alternating direction method of multipliers. We support our theoretical analysis with some preliminary computational experiments, which confirm the superiority of our method over other related ones in the literature.\\

\noindent
\textbf{Outline.}
In Section \ref{Sec:Prelims}, we give some basic definitions and results needed in subsequent sections. In Section \ref{Sec:main}, we derive our method from the dynamical systems and later introduce our proposed method. We also give both weak convergence and non-asymptotic $O(1/n)$ convergence rate of our method in Section \ref{Sec:main}. We give applications of our results to convex-concave saddle-point problems, the proximal method of multipliers, ADMM, primal–dual hybrid gradient method and Douglas–Rachford splitting method in Section \ref{Sec:Applications}.
We give some numerical illustrations in Section \ref{sec:numerics} and concluding remarks are given in Section \ref{conclude}.

\section{Preliminaries}\label{Sec:Prelims}
\noindent
In this section, we give some definitions and basic results that will
be used in our subsequent analysis. The weak and the strong convergence of $\{x_n\}\subset H$ to $x \in H$
is denoted by $x_n \rightharpoonup x$ and $x_n \to x$ as $n \to \infty $ respectively.

\begin{dfn}
A mapping $T: H \to H$ is called
\begin{itemize}
	\item[(i)] nonexpansive if $\|Tx - Ty \| \leq \|x-y \|,$ for all $x,y \in H;$
	\item[(ii)] firmly nonexpansive if $\|Tx - Ty\|^2 \leq \|x- y \|^2 - \|(I-T)x - (I - T)y \|^2$ for all $x,y \in H.$ Equivalently, $T$ is firmly nonexpansive if $\|Tx - Ty\|^2 \leq \langle x - y, Tx - Ty \rangle$ for all $x,y \in H;$
	\item[(iii)] averaged if $T$ can be expressed as the averaged of the identity mapping $I$ and a nonexpansive mapping $S$, i.e., $T = (1 - \alpha)I + \alpha S$ with $\alpha  \in (0,1).$ Alternatively, $T$ is $\alpha$-averaged if
$$
\|Tx-Ty\|^2 \leq \|x-y\|^2-\frac{1-\alpha}{\alpha}\|(I-T)x-(I-T)y\|^2, \forall x,y \in H.
$$
\end{itemize}
\end{dfn}

\begin{dfn}
A multivalued mapping $A:H \to 2^H$ is said to be monotone if for any $x,y \in H,$
\begin{equation}
\langle x - y, f - g \rangle \geq 0, \nonumber
\end{equation}
where $f \in Ax $ and $g \in Ay.$ The Graph of $A$ is defined by $$Gr(A):= \{(x,f)\in H \times H : f \in Ax\}.$$ If $Gr(A)$ is not properly contained in the graph of any other monotone mapping, then we say that $A$ is maximal.  It is well-known that for each $x \in H$, and
$\lambda>0$, there is a unique $z \in H$ such that $x \in (I+\lambda A)z$. The single-valued operator $J_{\lambda}^A(x)$ is called the resolvent of $A$ (see \cite{12}).
\end{dfn}

\begin{lem}\label{lm2}
The following identities hold for all $u,v,w\in H$:
$$2\langle u,v\rangle=\|u\|^{2}+\|v\|^{2}-\|u-v\|^{2}=\|u+v\|^{2}-\|u\|^{2}-\|v\|^{2}.$$
\end{lem}

\begin{lem} \label{simple}
Let $x,y,z \in H$ and $a,b \in \mathbb{R}$. Then
\begin{eqnarray*}
\|(1+a)x-(a-b)y-bz\|^2&=& (1+a)\|x\|^2-(a-b)\|y\|^2-b\|z\|^2 +(1+a)(a-b)\|x-y\|^2\\
&&+b(1+a)\|x-z\|^2-b(a-b)\|y-z\|^2.
\end{eqnarray*}
\end{lem}

\section{Main Results}\label{Sec:main}

\subsection{Motivations from Dynamical Systems}

\noindent Consider the following second order dynamical system
\begin{equation}\label{situ}
 \ddot{x}(t)+\alpha(t)\dot{x}(t)+\beta(t)(x(t)-J_{\lambda}^A(x(t)))=0,~~x(0)=x_0, \dot{x}(0)=v_0,
\end{equation}
where $\alpha,\beta:[0,\infty) \rightarrow [0,\infty)$ are Lebesgue measurable functions and $\lambda >0$. Let $0<\omega_2<\omega_1$ be two weighting parameters such that
$\omega_1+\omega_2=1$, $h>0$ is the time step-size, $t_n=nh$ and $x_n=x(t_n)$. Consider an explicit Euler forward discretization with respect to $J_{\lambda}^A$, explicit discretization of $\dot{x}(t)$, and a weighted sum of explicit and implicit discretization of $\ddot{x}(t)$, we have
\begin{eqnarray}\label{situ1}
&&\frac{\omega_1}{h^2}(x_{n+1}-2x_n+x_{n-1})+\frac{\omega_2}{h^2}(x_n-2x_{n-1}+x_{n-2})\nonumber \\
&&+\frac{\alpha_n}{h}(x_n-x_{n-1})+\beta_n(y_n-J_{\lambda}^A(y_n))=0,
\end{eqnarray}
where $y_n$ performs "extrapolation" onto the points $x_n, x_{n-1}$ and $x_{n-2}$, which will be chosen later. We
observe that since $T:=I-J_{\lambda}^A$ is Lipschitz continuous, there is some flexibility
in this choice. Therefore, \eqref{situ1} becomes
\begin{eqnarray}\label{situ2}
  &&x_{n+1}-2x_n+x_{n-1}+\frac{\omega_2}{\omega_1}(x_n-2x_{n-1}+x_{n-2})\nonumber \\
  &&+ \frac{\alpha_n h}{\omega_1}(x_n-x_{n-1})+\frac{\beta_nh^2}{\omega_1}(y_n-J_{\lambda}^A(y_n))=0.
\end{eqnarray}
This implies that
\begin{eqnarray}\label{situ3}
  &&x_{n+1}=2x_n-x_{n-1}-\frac{\omega_2}{\omega_1}(x_n-2x_{n-1}+x_{n-2}) \nonumber \\
 &&- \frac{\alpha_n h}{\omega_1}(x_n-x_{n-1})-\frac{\beta_nh^2}{\omega_1}(y_n-J_{\lambda}^A(y_n))\nonumber \\
&=&x_n+ (x_n-x_{n-1})-\frac{\omega_2}{\omega_1}(x_n-x_{n-1})+\frac{\omega_2}{\omega_1}(x_{n-1}-x_{n-2}) \nonumber \\
&&- \frac{\alpha_n h}{\omega_1}(x_n-x_{n-1})-\frac{\beta_nh^2}{\omega_1}(y_n-J_{\lambda}^A(y_n))\nonumber \\
&=& x_n+ \Big(1-\frac{\omega_2}{\omega_1}-\frac{\alpha_n h}{\omega_1}\Big)(x_n-x_{n-1})+ \frac{\omega_2}{\omega_1}(x_{n-1}-x_{n-2}) \nonumber \\
&&-\frac{\beta_nh^2}{\omega_1}(y_n-J_{\lambda}^A(y_n)).
\end{eqnarray}
Set
$$
\theta_n:=1-\frac{\omega_2}{\omega_1}-\frac{\alpha_n h}{\omega_1},~~\delta_n:=\frac{\omega_2}{\omega_1},~~ \rho_n:=\frac{\beta_nh^2}{\omega_1}.
$$
\noindent Then we have from \eqref{situ3} that
\begin{eqnarray}\label{situ4}
x_{n+1}&=&x_n+\theta_n(x_n-x_{n-1})+\delta_n(x_{n-1}-x_{n-2})\nonumber \\
  &&-\rho_ny_n+\rho_nJ_{\lambda}^A(y_n).
\end{eqnarray}
Choosing $y_n=x_n+\theta_n(x_n-x_{n-1})+\delta_n(x_{n-1}-x_{n-2})$, then \eqref{situ4} becomes
\begin{eqnarray}\label{situ5}
\left\{  \begin{array}{ll}
      & y_n=x_n+\theta_n(x_n-x_{n-1})+\delta_n(x_{n-1}-x_{n-2}),\\
      & x_{n+1}=(1-\rho_n)y_n+\rho_nJ_{\lambda}^A(y_n)
      \end{array}
      \right.
\end{eqnarray}
This is two-step inertial proximal point algorithm we intend to study in the next section of this paper.

\subsection{Proposed Method}\label{Sec:Method}
\noindent In this subsection, we consider two-step inertial proximal point algorithm given in \eqref{situ5} with $\theta_n=\theta,~~\delta_n=\delta$ and $\rho_n=1$ for the sake of simplicity.\\

\noindent
 In our convergence analysis, we assume that parameters $\delta$ and $\theta$ lie in the following region:

\begin{equation}\label{osi}
\mathcal{G}:=\Big\{(\delta,\theta): 0\leq \theta < \frac{1}{3}, \frac{3\theta-1}{3+4\theta} < \delta\leq 0\Big\}.
\end{equation}
\noindent One can see clearly from \eqref{osi} that $\delta < 1-3\theta$. \\


\noindent We now present our proposed method as follows:

\begin{algorithm}[H]
\caption{(2-Step Inertial PPA)}\label{alg1}
\begin{algorithmic}[1]
\State  Choose parameters $\delta$ and $\theta$ satisfying condition \eqref{osi}. Choose $x_{-1}, x_0, y_0 \in H$ arbitrarily, $\lambda >0$ and set $n=0.$
\State  Given $x_{n-1}, x_n$ and $y_n$, compute $x_{n+1}$ as follows:
        \begin{eqnarray}\label{3.1a}
\left\{  \begin{array}{ll}
      & x_{n+1}=J_{\lambda}^A(y_n),\\
      & y_{n+1}=x_{n+1}+\theta(x_{n+1}-x_n)+\delta(x_n-x_{n-1})
      \end{array}
      \right.
      \end{eqnarray}
\State Set $ n \leftarrow n+1 $, and {\bf go to Step 2}.
\end{algorithmic}
\end{algorithm}

\begin{rem}
When $\delta=0$ in our proposed Algorithm \ref{alg1}, our method reduces to the inertial proximal point algorithm studied in
\cite{4,3,AttouchCabot,6,Aujol,9,Chen,21,22} to mention but a few. Our method is an extension of the inertial proximal point algorithm in
\cite{4,3,AttouchCabot,6,Aujol,9,Chen,21,22}. We will show the advantage gained with the introduction of $\delta \in (-\infty,0]$ in the numerical experiments in Section \ref{sec:numerics}.
\end{rem}

\subsection{Convergence Analysis}
\noindent We present the weak convergence analysis of sequence of iterates generated by our proposed Algorithm \ref{alg1} in this subsection.

\begin{thm}\label{thm1}
	Let $A:H \to 2^H$ be a maximal monotone. Suppose $A^{-1}(\textbf{0}) \neq \emptyset$ and let $\{x_n\}$ be generated by Algorithm \ref{alg1}. Then $\{x_n\}$ converges weakly to a point in $A^{-1}(\textbf{0}).$
\end{thm}

\begin{proof}
Let $x^* \in A^{-1}(\textbf{0})$. Then
\begin{eqnarray*}
y_n-x^*&=&x_n+\theta(x_n-x_{n-1})+\delta(x_{n-1}-x_{n-2})-x^*\\
&=&(1+\theta)(x_n-x^*)-(\theta-\delta)(x_{n-1}-x^*)-\delta(x_{n-2}-x^*)
\end{eqnarray*}
Therefore, by Lemma \ref{simple}, we obtain
\begin{eqnarray}\label{ade1}
\|y_n-x^*\|^2 &=& \|(1+\theta)(x_n-x^*)-(\theta-\delta)(x_{n-1}-x^*)-\delta(x_{n-2}-x^*)\|^2 \nonumber \\
&=&(1+\theta)\|x_n-x^*\|^2-(\theta-\delta)\|x_{n-1}-x^*\|^2-\delta\|x_{n-2}-x^*\|^2\nonumber \\
&&+(1+\theta)(\theta-\delta)\|x_n-x_{n-1}\|^2+\delta(1+\theta)\|x_n-x_{n-2}\|^2\nonumber \\
&&-\delta(\theta-\delta)\|x_{n-1}-x_{n-2}\|^2.
\end{eqnarray}
 By Cauchy-Schwartz inequality, we obtain
\begin{equation}\label{happy1}
-2\theta \langle x_{n+1}-x_n, x_n-x_{n-1}\rangle \geq -2\theta \|x_{n+1}-x_n\|\|x_n-x_{n-1}\|,
\end{equation}
\begin{equation}\label{happy2}
-2\delta \langle x_{n+1}-x_n,x_{n-1}-x_{n-2} \rangle  \geq -2|\delta| \|x_{n+1}-x_n\|\|x_{n-1}-x_{n-2}\|,
\end{equation}
and
\begin{eqnarray}\label{happy3}
2\delta\theta \langle x_n-x_{n-1},x_{n-1}-x_{n-2}\rangle &=& -2 \delta\theta \langle x_{n-1}-x_n,x_{n-1}-x_{n-2}\rangle  \nonumber \\
&\geq& -2|\delta|\theta \|x_n-x_{n-1}\|\|x_{n-1}-x_{n-2}\|.
\end{eqnarray}
By \eqref{happy1}, \eqref{happy2} and \eqref{happy3}, we obtain
\begin{eqnarray}\label{ade2}
\|x_{n+1}-y_n\|^2 &=& \|x_{n+1}-(x_n+\theta(x_n-x_{n-1})+\delta(x_{n-1}-x_{n-2}))\|^2 \nonumber\\
&=&\|x_{n+1}-x_n-\theta(x_n-x_{n-1})-\delta(x_{n-1}-x_{n-2})\|^2 \nonumber\\
&=&\|x_{n+1}-x_n\|^2-2\theta \langle x_{n+1}-x_n, x_n-x_{n-1}\rangle \nonumber\\
&&-2\delta \langle x_{n+1}-x_n,x_{n-1}-x_{n-2} \rangle+\theta^2\|x_n-x_{n-1}\|^2\nonumber\\
&&+2\delta\theta \langle x_n-x_{n-1},x_{n-1}-x_{n-2}\rangle+\delta^2\|x_{n-1}-x_{n-2}\|^2\nonumber\\
&\geq& \|x_{n+1}-x_n\|^2-2\theta \|x_{n+1}-x_n\|\|x_n-x_{n-1}\|\nonumber\\
&&-2|\delta| \|x_{n+1}-x_n\|\|x_{n-1}-x_{n-2}\|+\theta^2\|x_n-x_{n-1}\|^2\nonumber\\
&&-2|\delta|\theta \|x_n-x_{n-1}\|\|x_{n-1}-x_{n-2}\|+\delta^2\|x_{n-1}-x_{n-2}\|^2\nonumber\\
&\geq& \|x_{n+1}-x_n\|^2-\theta \|x_{n+1}-x_n\|^2-\theta\|x_n-x_{n-1}\|^2\nonumber\\
&&-|\delta| \|x_{n+1}-x_n\|^2-|\delta|\|x_{n-1}-x_{n-2}\|^2+\theta^2\|x_n-x_{n-1}\|^2\nonumber\\
&&-|\delta|\theta \|x_n-x_{n-1}\|^2-|\delta|\theta\|x_{n-1}-x_{n-2}\|^2+\delta^2\|x_{n-1}-x_{n-2}\|^2\nonumber\\
&=&(1-|\delta|-\theta)\|x_{n+1}-x_n\|^2+(\theta^2-\theta-|\delta|\theta)\|x_n-x_{n-1}\|^2\nonumber\\
&&+(\delta^2-|\delta|-|\delta|\theta)\|x_{n-1}-x_{n-2}\|^2.
\end{eqnarray}
Using \eqref{ade1} and \eqref{ade2} in \eqref{3.1a}, we have (noting that $J_{\lambda}^A$ is firmly nonexpansive and $\delta \leq 0$)
\begin{eqnarray}\label{ade3}
\|x_{n+1}-x^*\|^2&=&\|J_{\lambda}^A(y_n)-x^*\|^2 \nonumber \\
&=&  \|J_{\lambda}^A(y_n)-J_{\lambda}^A(x^*)\|^2 \nonumber \\
&\leq&  \|y_n-x^*\|^2-\|J_{\lambda}^A(y_n)-y_n\|^2 \nonumber \\
&=& \|y_n-x^*\|^2 -\|x_{n+1}-y_n\|^2 \nonumber\\
&\leq&(1+\theta)\|x_n-x^*\|^2-(\theta-\delta)\|x_{n-1}-x^*\|^2-\delta\|x_{n-2}-x^*\|^2\nonumber \\
&&+(1+\theta)(\theta-\delta)\|x_n-x_{n-1}\|^2+\delta(1+\theta)\|x_n-x_{n-2}\|^2\nonumber \\
&&-\delta(\theta-\delta)\|x_{n-1}-x_{n-2}\|^2-(1-|\delta|-\theta)\|x_{n+1}-x_n\|^2\nonumber \\
&&-(\theta^2-\theta-|\delta|\theta)\|x_n-x_{n-1}\|^2-(\delta^2-|\delta|-|\delta|\theta)\|x_{n-1}-x_{n-2}\|^2\nonumber \\
&=& (1+\theta)\|x_n-x^*\|^2-(\theta-\delta)\|x_{n-1}-x^*\|^2-\delta\|x_{n-2}-x^*\|^2\nonumber \\
&&+\Big((1+\theta)(\theta-\delta)-(\theta^2-\theta-|\delta|\theta)\Big)\|x_n-x_{n-1}\|^2\nonumber \\
&&+\delta(1+\theta)\|x_n-x_{n-2}\|^2-(1-|\delta|-\theta)\|x_{n+1}-x_n\|^2\nonumber \\
&&-\Big(\delta(\theta-\delta)+(\delta^2-|\delta|-|\delta|\theta)\Big)\|x_{n-1}-x_{n-2}\|^2\nonumber \\
&=& (1+\theta)\|x_n-x^*\|^2-(\theta-\delta)\|x_{n-1}-x^*\|^2-\delta\|x_{n-2}-x^*\|^2\nonumber \\
&&+(2\theta-\delta-\delta\theta+|\delta|\theta)\|x_n-x_{n-1}\|^2+\delta(1+\theta)\|x_n-x_{n-2}\|^2\nonumber \\
&&-(1-|\delta|-\theta)\|x_{n+1}-x_n\|^2+(|\delta|+|\delta|\theta-\delta\theta) \|x_{n-1}-x_{n-2}\|^2\nonumber \\
&\leq&
(1+\theta)\|x_n-x^*\|^2-(\theta-\delta)\|x_{n-1}-x^*\|^2-\delta\|x_{n-2}-x^*\|^2\nonumber \\
&&+(2\theta-\delta-\delta\theta+|\delta|\theta)\|x_n-x_{n-1}\|^2-(1-|\delta|-\theta)\|x_{n+1}-x_n\|^2\nonumber \\
&&+(|\delta|+|\delta|\theta-\delta\theta) \|x_{n-1}-x_{n-2}\|^2.
\end{eqnarray}
Therefore,
\begin{eqnarray}\label{ade4}
&& \|x_{n+1}-x^*\|^2-\theta\|x_n-x^*\|^2-\delta\|x_{n-1}-x^*\|^2+(1-|\delta|-\theta)\|x_{n+1}-x_n\|^2\nonumber \\
&\leq& \|x_n-x^*\|^2-\theta\|x_{n-1}-x^*\|^2-\delta\|x_{n-2}-x^*\|^2 \nonumber \\
&&+(1-|\delta|-\theta)\|x_n-x_{n-1}\|^2+(3\theta-1+(1+\theta)(|\delta|-\delta))\|x_n-x_{n-1}\|^2 \nonumber \\
&&+ (|\delta|+|\delta|\theta-\delta\theta)\|x_{n-1}-x_{n-2}\|^2.
\end{eqnarray}
Now, define
$$
\Gamma_{n}:=\|x_n-x^*\|^2-\theta\|x_{n-1}-x^*\|^2-\delta\|x_{n-2}-x^*\|^2+(1-|\delta|-\theta)\|x_n-x_{n-1}\|^2, ~~n\geq 1.
$$
\noindent
We show that $\Gamma_n \geq 0, ~~\forall n \geq 1.$ Observe that
\begin{eqnarray}\label{ade12}
\|x_{n-1}-x^*\|^2\leq 2\|x_n-x_{n-1}\|^2 +2\|x_n-x^*\|^2.
\end{eqnarray}
So,
\begin{eqnarray}\label{ade13}
\Gamma_n&=&  \|x_n-x^*\|^2-\theta\|x_{n-1}-x^*\|^2-\delta\|x_{n-2}-x^*\|^2\nonumber \\
&&+(1-|\delta|-\theta)\|x_n-x_{n-1}\|^2\nonumber \\
&\geq & \|x_n-x^*\|^2-2\theta\|x_n-x_{n-1}\|^2-2\theta\|x_n-x^*\|^2\nonumber \\
&&-\delta\|x_{n-2}-x^*\|^2+(1-|\delta|-\theta)\|x_n-x_{n-1}\|^2\nonumber \\
&=& (1-2\theta)\|x_n-x^*\|^2+(1-|\delta|-3\theta)\|x_n-x_{n-1}\|^2\nonumber \\
&&-\delta\|x_{n-2}-x^*\|^2\nonumber \\
&\geq&0,
\end{eqnarray}
since $|\delta|< 1-3\theta\Leftrightarrow 3\theta-1 <\delta< 1-3\theta$ and $\delta \leq 0$. Furthermore, we obtain from \eqref{ade4} that
\begin{eqnarray}\label{ade5}
\Gamma_{n+1}-\Gamma_{n}& \leq & -(3\theta-1+(1+\theta)(|\delta|-\delta))(\|x_{n-1}-x_{n-2}\|^2-\|x_n-x_{n-1}\|^2) \nonumber \\
&&-\Big[-(3\theta-1+(1+\theta)(|\delta|-\delta))-(|\delta|+|\delta|\theta-\delta\theta) \Big]\|x_{n-1}-x_{n-2}\|^2 \nonumber \\
& = & -(3\theta-1+(1+\theta)(|\delta|-\delta))(\|x_{n-1}-x_{n-2}\|^2-\|x_n-x_{n-1}\|^2) \nonumber \\
&&-\Big(1-3\theta-2|\delta|-2\theta|\delta|+2\theta\delta+\delta\Big)\|x_{n-1}-x_{n-2}\|^2\nonumber \\
& = &  c_1(\|x_{n-1}-x_{n-2}\|^2-\|x_n-x_{n-1}\|^2)-c_2\|x_{n-1}-x_{n-2}\|^2,
\end{eqnarray}
 where $ c_1:=-(3\theta-1+(1+\theta)(|\delta|-\delta))$ and  $c_2:=1-3\theta-2|\delta|-2\theta|\delta|+2\theta\delta+\delta.$
Noting that $|\delta|=-\delta$ since $\delta\leq 0$, we then have that
$$
c_1=-(3\theta-1+(1+\theta)(|\delta|-\delta)) > 0 \Leftrightarrow \frac{3\theta-1}{2(1+\theta)} < \delta.
$$
\noindent Furthermore,
$$
c_2=1-3\theta-2|\delta|-2\theta|\delta|+2\theta\delta+\delta > 0 \Leftrightarrow \frac{3\theta-1}{3+4\theta} < \delta.
$$
\noindent Observe that if $0\leq \theta < \frac{1}{3}$, then
$$
3\theta-1 <  \frac{3\theta-1}{2(1+\theta)} < \frac{3\theta-1}{3+4\theta}.
$$
\noindent
This implies by \eqref{osi} that both $c_1:=-(3\theta-1+(1+\theta)(|\delta|-\delta))>0$ and $c_2:=1-3\theta-2|\delta|-2\theta|\delta|+2\theta\delta+\delta > 0$ if
\begin{equation}\label{ade6}
\frac{3\theta-1}{3+4\theta} < \delta \leq  0.
\end{equation}
By \eqref{ade5}, we obtain
\begin{eqnarray}\label{ade7}
\Gamma_{n+1}+c_1\|x_n-x_{n-1}\|^2&\leq& \Gamma_{n}+c_1\|x_{n-1}-x_{n-2}\|^2\nonumber \\
&&-c_2\|x_{n-1}-x_{n-2}\|^2.
\end{eqnarray}
Letting $\bar{\Gamma}_n:=\Gamma_n+c_1\|x_{n-1}-x_{n-2}\|^2$, we obtain from \eqref{ade7} that
\begin{equation}\label{afikun}
 \bar{\Gamma}_{n+1} \leq \bar{\Gamma}_{n}.
\end{equation}
This implies from \eqref{afikun} that the sequence $\{\bar{\Gamma}_{n}\}$ is decreasing and thus $\underset{n\rightarrow \infty}\lim \bar{\Gamma}_{n}$ exists. Consequently, we have from \eqref{ade7} that
\begin{eqnarray}\label{ade9}
\underset{n\rightarrow \infty}\lim  c_2\|x_{n-1}-x_{n-2}\|^2=0.
	\end{eqnarray}
Hence,
\begin{eqnarray}\label{ade10}
\underset{n\rightarrow \infty}\lim \|x_{n-1}-x_{n-2}\|=0.
	\end{eqnarray}
Using \eqref{ade10} and existence of limit of $\{\bar{\Gamma}_{n}\}$, we have that $\underset{n\rightarrow \infty}\lim \Gamma_n$ exists. Furthermore,
\begin{eqnarray}\label{ade14}
\|x_{n+1}-y_n\|& =& \|x_{n+1}-x_n-\theta(x_n-x_{n-1})-\delta(x_{n-1}-x_{n-2})\|\nonumber \\
&\leq& \|x_{n+1}-x_n\|+\theta\|x_n-x_{n-1}\|+|\delta|\|x_{n-1}-x_{n-2}\|\rightarrow 0
\end{eqnarray}
as $n\rightarrow \infty.$
Therefore,
\begin{eqnarray}\label{ade15}
\underset{n\rightarrow \infty}\lim \|J_{\lambda}^A(y_n)-y_n\|=0.
	\end{eqnarray}
Also,
\begin{eqnarray}\label{ade16}
\|y_n-x_n\|\leq \theta\|x_n-x_{n-1}\|+|\delta|\|x_{n-1}-x_{n-2}\|\rightarrow 0, n\rightarrow \infty.
	\end{eqnarray}
Since $\underset{n\rightarrow \infty}\lim \Gamma_n$ exists and $\underset{n\rightarrow \infty}\lim \|x_n-x_{n-1}\|=0$, we obtain from \eqref{ade13} that the sequence $\{x_n\}$ is bounded. Hence $\{x_n\}$ has at least one accumulation point $v^* \in H$. Assume that $\{x_{n_k}\} \subset \{x_{n}\}$ such that $x_{n_k} \rightharpoonup v^*,$ $k \to \infty$. Since $y_n-x_n\rightarrow 0, n\rightarrow \infty$, we have that $y_{n_k}\rightharpoonup v^*$, $k \to \infty$. Also, by \eqref{ade15}, we have that $J_{\lambda}^A(y_{n_k}) \rightharpoonup v^*,$ $k \to \infty$.
Since $J_{\lambda}^A (y_n)=(I+\lambda A)^{-1}y_n$ we obtain
$$
\frac{1}{\lambda}\Big(y_{n_k}-J_{\lambda}^A(y_{n_k})\Big) \in A(J_{\lambda}^A(y_{n_k})).
$$
\noindent By the monotonicity of $A$, we have $\forall n_k$,
\begin{equation}\label{ade17}
\langle x- J_{\lambda}^A(y_{n_k}), w-\frac{1}{\lambda}\Big(y_{n_k}-J_{\lambda}^A(y_{n_k})\Big) \rangle \geq 0,\  \forall \ x, w \  \mbox{satisfying} \ w\in A(x).
\end{equation}
Using  \eqref{ade15} in \eqref{ade17}  we get
$\langle x-v^*, w \rangle \geq 0, \forall x,w$ satisfying $w\in A(x)$. Since $A$  is maximal monotone, we conclude that $v^* \in A^{-1}(\textbf{0})$ (see, for example, \cite{Rockafellar27}).\\

\noindent
Since $\underset{n\rightarrow \infty}\lim \Gamma_{n}$ exists and $\underset{n\rightarrow \infty}\lim \|x_{n+1}-x_n\|=0$, we have that
\begin{eqnarray}\label{omilomi2}
\underset{n\rightarrow \infty}\lim \Big[\|x_n-x^*\|^2-\theta\|x_{n-1}-x^*\|^2-\delta\|x_{n-2}-x^*\|^2\Big]
\end{eqnarray}
exists.\\

\noindent We now show that $x_n\rightharpoonup x^* \in A^{-1}(\textbf{0})$. Let us assume that there exist $\{x_{n_k}\} \subset \{x_n\}$ and $\{x_{n_j}\} \subset \{x_n\}$ such that $x_{n_k}\rightharpoonup v^*, k\rightarrow \infty$ and $x_{n_j}\rightharpoonup x^*, j\rightarrow \infty$. We show that $v^*=x^*$.\\

\noindent Observe that
\begin{equation}\label{away1}
2\langle x_n,x^*-v^*\rangle =\|x_n-v^*\|^2-\|x_n-x^*\|^2-\|v^*\|^2+\|x^*\|^2,
\end{equation}

\begin{equation}\label{away2}
2\langle x_{n-1},x^*-v^*\rangle =\|x_{n-1}-v^*\|^2-\|x_{n-1}-x^*\|^2-\|v^*\|^2+\|x^*\|^2,
\end{equation}
and
\begin{equation}\label{away2a}
2\langle x_{n-2},x^*-v^*\rangle =\|x_{n-2}-v^*\|^2-\|x_{n-2}-x^*\|^2-\|v^*\|^2+\|x^*\|^2,
\end{equation}
Therefore,
\begin{eqnarray}\label{away3}
2\langle -\theta x_{n-1},x^*-v^*\rangle &=&-\theta\|x_{n-1}-v^*\|^2+\theta\|x_{n-1}-x^*\|^2\nonumber \\
&&+\theta\|v^*\|^2-\theta\|x^*\|^2
\end{eqnarray}
and
\begin{eqnarray}\label{away3a}
2\langle -\delta x_{n-2},x^*-v^*\rangle &=&-\delta\|x_{n-2}-v^*\|^2+\delta\|x_{n-2}-x^*\|^2\nonumber \\
&&+\delta\|v^*\|^2-\delta\|x^*\|^2
\end{eqnarray}
Addition of \eqref{away1}, \eqref{away3} and \eqref{away3a} gives
\begin{eqnarray*}
2\langle x_n-\theta x_{n-1}-\delta x_{n-2},x^*-v^*\rangle&=& \Big(\|x_n-v^*\|^2-\theta \|x_{n-1}-v^*\|^2-\delta \|x_{n-2}-v^*\|^2 \Big) \\
  &&-\Big(\|x_n-x^*\|^2-\theta \|x_{n-1}-x^*\|^2-\delta \|x_{n-2}-x^*\|^2 \Big)\\
  &&+(1-\theta-\delta)(\|x^*\|^2-\|v^*\|^2).
\end{eqnarray*}
According to \eqref{omilomi2}, we have
\begin{eqnarray*}
\underset{n\rightarrow \infty}\lim \Big[\|x_n-x^*\|^2 -\theta\|x_{n-1}-x^*\|^2-\delta \|x_{n-2}-x^*\|^2\Big]
\end{eqnarray*}
exists and
\begin{eqnarray*}
\underset{n\rightarrow \infty}\lim \Big[\|x_n-v^*\|^2 -\theta\|x_{n-1}-v^*\|^2-\delta \|x_{n-2}-v^*\|^2 \Big]
\end{eqnarray*}
exists. This implies that
$$
\underset{n\rightarrow \infty}\lim \langle x_n-\theta x_{n-1}-\delta x_{n-2},x^*-v^*\rangle
$$
\noindent exists. Now,
\begin{eqnarray*}
 \langle v^*-\theta v^*-\delta v^*,x^*-v^*\rangle&=& \underset{k\rightarrow \infty}\lim \langle x_{n_k}-\theta x_{n_k-1}-\delta x_{n_k-2},x^*-v^*\rangle  \\
  &=& \underset{n\rightarrow \infty}\lim \langle x_n-\theta x_{n-1}-\delta x_{n-2},x^*-v^*\rangle \\
&=& \underset{j\rightarrow \infty}\lim \langle x_{n_j}-\theta x_{n_j-1}-\delta x_{n_j-2},x^*-v^*\rangle  \\
&=& \langle x^*-\theta x^*-\delta x^*,x^*-v^*\rangle,
\end{eqnarray*}
and this yields
$$
(1-\theta-\delta)\|x^*-v^*\|^2=0.
$$
\noindent Since $\delta \leq 0< 1-\theta$, we obtain that $x^*=v^*$. Hence, $\{x_n\}$ converges weakly to a point in $A^{-1}(0)$. This completes the proof.

\end{proof}

\begin{rem}\label{Rem:DiscussionWeakConvergence}
(a) The conditions imposed on the parameters in our proposed Algorithm \ref{alg1} are weaker  than the ones imposed in
\cite[4.1 Algorithms]{DongJOGO} and \cite[Chapter 4, (4.2.5)]{Liang}. For example, we do not impose the summability conditions in  \cite[Theorem 4.1 (32)]{DongJOGO}, \cite[Theorem 4.2 (35)]{DongJOGO} and \cite[Chapter 4 (4.2.5)]{Liang} in our Algorithm \ref{alg1}. Therefore, our results are improvements over the results obtained in \cite{DongJOGO}.\\[-1mm]

\noindent
(b) The assumptions on the parameters in Algorithm \ref{alg1} are also different from the conditions imposed on the iterative parameters in
\cite[Algorithm 1.2]{CombettesGlaudin}. For example, conditions (b) and (c) of \cite[Algorithm 1.2]{CombettesGlaudin} are not needed in our convergence analysis. More importantly, we do not assume that $\theta+\delta=1$ (as imposed in \cite[Algorithm 1.2]{CombettesGlaudin}) in our convergence analysis \hfill $\Diamond$
\end{rem}

\noindent
In the next result, we give a non-asymptotic $O(1/n)$ convergence rate of our proposed Algorithm \ref{alg1}.

\begin{thm}\label{thm2}
	Let $A:H \to 2^H$ be a maximal monotone. Assume that $A^{-1}(\textbf{0})\neq \emptyset$ and $x_0=x_{-1}=x_{-2}$. Let $\{x_n\}$ be generated by Algorithm \ref{alg1}. Then, for any $x^* \in A^{-1}(\textbf{0})$ and $n>0$, it holds that
\begin{equation}\label{ser}
\underset{0 \leq j \leq n-2}\min
\|x_{j+1}-y_j\|^2 \leq 3\Big(1+\theta^2+\delta^2\Big)\frac{1}{c_2}\frac{(1-\theta-\delta)\|x_0-x^*\|^2}{n-1}
\end{equation}
where $c_2=1-3\theta-2|\delta|-2\theta|\delta|+2\theta\delta+\delta$.
\end{thm}

\begin{proof}
Let $x^* \in A^{-1}(0)$. Observe that
\begin{eqnarray}\label{lag1}
\bar{\Gamma}_{0}&=&\Gamma_{0}=\|x_0-x^*\|^2-\theta\|x_{-1}-x^*\|^2-\delta\|x_{-2}-x^*\|^2+(1-|\delta|-\theta)\|x_0-x_{-1}\|^2\nonumber \\
         &=& \|x_0-x^*\|^2-\theta\|x_{-1}-x^*\|^2-\delta\|x_{-2}-x^*\|^2\nonumber\\
          &=& (1-\theta-\delta)\|x_0-x^*\|^2.
\end{eqnarray}
From \eqref{ade7}, we obtain
\begin{eqnarray}\label{lag2}
c_2\sum_{j=0}^{n} \|x_{j-1}-x_{j-2}\|^2 \leq \bar{\Gamma}_0-\bar{\Gamma}_{n+1}.
\end{eqnarray}
This implies that
\begin{eqnarray}\label{lag2a}
\sum_{j=0}^{n} \|x_{j-1}-x_{j-2}\|^2 &\leq& \frac{1}{c_2} \bar{\Gamma}_0=\frac{1}{c_2}\Gamma_0 \nonumber \\
&=&\frac{1}{c_2}(1-\theta-\delta)\|x_0-x^*\|^2.
\end{eqnarray}
This implies that
\begin{eqnarray}\label{lag2aa}
\underset{0 \leq j \leq n}\min\|x_{j-1}-x_{j-2}\|^2 \leq \frac{1}{c_2}\frac{(1-\theta-\delta)\|x_0-x^*\|^2}{n+1}.
\end{eqnarray}
Consequently,
\begin{eqnarray}\label{lag2aa1}
\underset{0 \leq j \leq n-2}\min\|x_{j+1}-x_j\|^2 \leq \frac{1}{c_2}\frac{(1-\theta-\delta)\|x_0-x^*\|^2}{n-1}
\end{eqnarray}
and
\begin{eqnarray}\label{lag2aa2}
\underset{0 \leq j \leq n-1}\min\|x_{j}-x_{j-1}\|^2 \leq \frac{1}{c_2}\frac{(1-\theta-\delta)\|x_0-x^*\|^2}{n}.
\end{eqnarray}
From \eqref{ade14}, we obtain
\begin{eqnarray}\label{osi2}
\|x_{n+1}-y_n\|^2& \leq & \Big(\|x_{n+1}-x_n\|+\theta\|x_n-x_{n-1}\|+|\delta|\|x_{n-1}-x_{n-2}\|\Big)^2\nonumber \\
&\leq & 3\Big(\|x_{n+1}-x_n\|^2+\theta^2\|x_n-x_{n-1}\|^2+\delta^2\|x_{n-1}-x_{n-2}\|^2\Big).
\end{eqnarray}
Therefore,
\begin{eqnarray}\label{lag3}
\underset{0 \leq j \leq n-2}\min
\|x_{j+1}-y_j\|^2 &\leq& 3 \underset{0 \leq j \leq n-2}\min \|x_{j+1}-x_j\|^2+3\theta^2 \underset{0 \leq j \leq n-2}\min \|x_j-x_{j-1}\|^2\nonumber \\
&&+3\delta^2 \underset{0 \leq j \leq n-2}\min\|x_{j-1}-x_{j-2}\|^2\nonumber \\
&\leq& \Big(3+3\theta^2+3\delta^2\Big)\frac{1}{c_2}\frac{(1-\theta-\delta)\|x_0-x^*\|^2}{n-1}.
\end{eqnarray}
This completes the proof.
\end{proof}

\begin{rem}
The non-asymptotic $O(1/n)$ convergence rate of Algorithm \ref{alg1} given in Theorem \ref{thm2} extends the convergence rate given in
\cite[Theorem 2, (2.13)]{Chen} for the case when $\delta=0$ in Algorithm \ref{alg1}.
\end{rem}

\section{Applications}\label{Sec:Applications}
\noindent
We give some applications of our proposed two-step inertial proximal point Algorithm \ref{alg1} to solving
convex-concave saddle-point problem using two-step inertial versions of augmented Lagrangian, the proximal method of multipliers, and ADMM. We also consider two-step inertial versions primal–dual hybrid gradient method and Douglas–Rachford splitting method.

\subsection{Convex–Concave Saddle-Point Problem}\label{sec:appl1}
\noindent
Let us consider the following convex-concave saddle-point problem:
\begin{equation}\label{convex1}
\min_{u\in H_1}\max_{v \in H_2} \phi(u,v),
\end{equation}
where $H_1$ and $H_2$ are real Hilbert spaces equipped with inner product $\langle.,.\rangle$ and
$\phi(.,v)\in \mathcal{F}(H_1), -\phi(u,.)\in \mathcal{F}(H_2)$, where the associated
saddle subdifferential operator
\begin{equation}\label{convex2}
\left(
  \begin{array}{c}
    \partial_u \phi(u,v) \\
    \partial_v(-\phi(u,v) \\
  \end{array}
\right)
\end{equation}
is monotone. Using the ideas in \cite{Rockafellar27}, we apply our proposed Algorithm \ref{alg1} to solve problem \eqref{convex1}
below.

\begin{algorithm}[H]
\caption{2-Step Inertial PPA for Saddle-Point Problem}\label{alg2}
\begin{algorithmic}[1]
\State  Choose parameters $\delta$ and $\theta$ satisfying condition \eqref{osi}. Choose $\phi(.,v)\in \mathcal{F}(H_1), -\phi(u,.)\in \mathcal{F}(H_2),  \hat{u}_0 \in H_1, \hat{v}_0 \in H_2, x_{-1}=x_0=y_0=(\hat{u}_0,\hat{v}_0)$ and $\lambda>0$. Set $n=0.$
\State  Compute as follows:
        \begin{eqnarray}
\left\{  \begin{array}{ll}
      & x_{n+1}=(u_{n+1},v_{n+1})={\rm arg}\min\limits_{u \in H_1} \max\limits_{v \in H_2} \Big\{\phi(u,v)+\frac{1}{2\lambda}\|u-\hat{u}_n\|^2-\frac{1}{2\lambda}\|v-\hat{v}_n\|^2 \Big\},\\
      & y_{n+1}=(\hat{u}_{n+1},\hat{v}_{n+1})=x_{n+1}+\theta(x_{n+1}-x_n)+\delta(x_n-x_{n-1})
      \end{array}
      \right.
      \end{eqnarray}
\State Set $ n \leftarrow n+1 $, and {\bf go to 2}.
\end{algorithmic}
\end{algorithm}

\noindent Furthermore, let us consider the following convex–concave
Lagrangian problem:

\begin{equation}\label{convex3}
\min_{u\in H_1}\max_{v\in H_2}\{L(u,v):f(u)+\langle v,Au-b\rangle \},
\end{equation}
which is associated with the following linearly constrained problem
\begin{eqnarray}\label{convex4}
\left\{  \begin{array}{ll}
      & \min\limits_{u \in H_1} f(u)\\
      & s.t. Au=b,
      \end{array}
      \right.
\end{eqnarray}
where $A \in B(H_1,H_2)$ and $b \in H_2$. Applying our proposed Algorithm \ref{alg1} to solving problem \eqref{convex3} gives

\begin{algorithm}[H]
\caption{2-Step Inertial Proximal Method of Multipliers}\label{alg3}
\begin{algorithmic}[1]
\State  Choose parameters $\delta$ and $\theta$ satisfying condition \eqref{osi}. Choose $\phi(.,v)\in \mathcal{F}(H_1), -\phi(u,.)\in \mathcal{F}(H_2), \hat{u}_0 \in H_1, \hat{v}_0 \in H_2,
 x_{-1}=x_0=y_0=(\hat{u}_0,\hat{v}_0)$ and $\lambda>0$. Set $n=0.$
\State  Compute as follows:
        \begin{eqnarray}
\left\{  \begin{array}{ll}
      &u_{n+1}={\rm arg}\min\limits_{u \in H_1}\Big\{L(u,\hat{v}_n)+\frac{\lambda}{2}\|Au-b\|^2+\frac{1}{2\lambda}\|u-\hat{u}_n\|^2 \Big\},\\
      & x_{n+1}=(u_{n+1},\hat{v}_n+\lambda(Au_{n+1}-b)),\\
      & y_{n+1}=(\hat{u}_{n+1},\hat{v}_{n+1})=x_{n+1}+\theta(x_{n+1}-x_n)+\delta(x_n-x_{n-1})
      \end{array}
      \right.
      \end{eqnarray}
\State Set $ n \leftarrow n+1 $, and {\bf go to 2}.
\end{algorithmic}
\end{algorithm}

\subsection{Version of Primal–Dual Hybrid Gradient Method}\label{sec:appl2}
\noindent In this section, let us consider the following coupled convex-concave saddle-point problem

\begin{equation}\label{convex5}
\min_{u\in H_1}\max_{v\in H_2}\{\phi(u,v)\equiv f(u)+\langle Ku,v\rangle-g(v)\},
\end{equation}
where $f \in \mathcal{F}(H_1), g \in \mathcal{F}(H_2)$ and $K \in \mathcal{B}(H_1,H_2)$. The primal-dual hybrid gradient (PDHG) method (see, for example, \cite{ChamPo,Esser}) is one of the popular methods for solving problem \eqref{convex5} and it is a preconditioned proximal point algorithm with $\lambda=1$ for the saddle subdifferential operator of $\phi$ given in \eqref{convex2} (see, \cite{ChambollePock,He}). Given the associated preconditioner as

\begin{equation}\label{convex6}
P= \left(
  \begin{array}{cc}
    \frac{1}{\tau}I & -K^* \\
    -K & \frac{1}{\sigma}I \\
  \end{array}
\right)
\end{equation}
which is positive definite when $\tau\sigma\|K\|^2<1$, $\|K\|:=\sup_{\|x\|\leq 1}\|Kx\|$; we adapt our proposed Algorithm \ref{alg1} to solve problem \eqref{convex5} as given below:

\begin{algorithm}[H]
\caption{2-Step Inertial PDHG Method}\label{alg4}
\begin{algorithmic}[1]
\State  Choose parameters $\delta$ and $\theta$ satisfying condition \eqref{osi}. Choose $f\in \mathcal{F}(H_1), g\in \mathcal{F}(H_2), K \in \mathcal{B}(H_1,H_2), \hat{u}_0 \in H_1, \hat{v}_0 \in H_2,
\tau\sigma\|K\|^2<1,  x_{-1}=x_0=y_0=(\hat{u}_0,\hat{v}_0)$. Set $n=0.$
\State  Compute as follows:
        \begin{eqnarray}
\left\{  \begin{array}{ll}
      &u_{n+1}={\rm arg}\min\limits_{u \in H_1}\Big\{f(u)+ \langle Ku,\hat{v}_n\rangle+ \frac{1}{2\tau}\|u-\hat{u}_n\|^2 \Big\},\\
      & v_{n+1}={\rm arg}\min\limits_{v \in H_2}\Big\{g(v)-\langle K(2u_{n+1}-\hat{u}_n),v\rangle+ \frac{1}{2\sigma}\|v-\hat{v}_n\|^2 \Big\},\\
      & x_{n+1}=(u_{n+1},v_{n+1}),\\
      & y_{n+1}=(\hat{u}_{n+1},\hat{v}_{n+1})=x_{n+1}+\theta(x_{n+1}-x_n)+\delta(x_n-x_{n-1})
      \end{array}
      \right.
      \end{eqnarray}
\State Set $ n \leftarrow n+1 $, and {\bf go to 2}.
\end{algorithmic}
\end{algorithm}

\subsection{Version of Douglas–Rachford Splitting Method}\label{sec:appl3}
\noindent This section presents an application of Algorithm \ref{alg1} to the
Douglas-Rachford splitting method for finding zeros of an operator
$ T $ such that $ T $ is the sum of two maximal monotone operators,
i.e. $ T = A + B $ with $ A, B : H \to 2^H $ being maximal monotone
multi-functions on a Hilbert space $ H $. The method was originally
introduced in \cite{DouglasRachford} in a finite-dimensional setting,
its extension to maximal monotone mappings in Hilbert spaces can be
found in \cite{LionsMercier}.\\

\noindent
The Douglas–Rachford splitting method \cite{DouglasRachford,LionsMercier} iteratively applies the operator
\begin{equation}\label{convex8}
G_{\lambda,A,B}:=J_{\lambda}^A o(2J_{\lambda}^B-I)+(I-J_{\lambda}^B)
\end{equation}
and has found to be effective in many applications including ADMM. In  \cite[Theorem 4]{12}, the Douglas–Rachford operator \eqref{convex8} was found to be a resolvent $J^M_{\lambda,A,B}$ of a maximal monotone operator
\begin{equation}\label{convex9}
M_{\lambda,A,B}:=G^{-1}_{\lambda,A,B}-I.
\end{equation}
Therefore, the Douglas–Rachford splitting method is a special case of the proximal
point method (with $\lambda=1$), written as
\begin{equation}\label{convex10}
v_{n+1}= J^M_{\lambda,A,B}(v_n)=G_{\lambda,A,B}(v_n).
\end{equation}
Hence, we can apply our Algorithm \ref{alg1} to the Douglas-Rachford splitting method as given below:

\begin{algorithm}[H]
\caption{2-Step Inertial Douglas-Rachford Splitting Method}\label{alg5}
\begin{algorithmic}[1]
\State  Choose parameters $\delta$ and $\theta$ satisfying condition \eqref{osi}. Choose $u_0=v_0=v_{-1} \in H, \lambda>0$. Set $n=0.$
\State  Compute as follows:
        \begin{eqnarray}
\left\{  \begin{array}{ll}
      &v_{n+1}=G_{\lambda,A,B}(u_n),\\
      & u_{n+1}=v_{n+1}+\theta(v_{n+1}-v_n)+\delta(v_n-v_{n-1})
      \end{array}
      \right.
      \end{eqnarray}
\State Set $ n \leftarrow n+1 $, and {\bf go to 2}.
\end{algorithmic}
\end{algorithm}

\subsection{Version of Alternating Direction Method of Multipliers (ADMM)}\label{sec:appl4}
\noindent
Suppose $H_1, H_2$ and$G$ are real Hilbert spaces and let us consider the following linearly constrained convex problem

\begin{eqnarray}\label{convex11}
\left\{  \begin{array}{ll}
      & \min\limits_{x \in H_1,z\in H_2} f(x)+g(z)\\
      & s.t. Ax+Bz=c,
      \end{array}
      \right.
\end{eqnarray}
where $f \in \mathcal{F}(H_1), g \in \mathcal{F}(H_2), A \in \mathcal{B}(H_1,G), B \in \mathcal{B}(H_2,G)$ and $c \in G$. The dual of problem \eqref{convex11} is given by
\begin{equation}\label{convex12}
\max_{v\in G}\{-f^*(-A^*v)-g^*(-B^*v)+\langle c,v\rangle\},
\end{equation}
where the conjugate of $f$ is given by $f^*(y):=\sup_{x\in H_1}\{\langle y,x\rangle-f(x)\}$ and the conjugate of $g$ is given by
$g^*(y):=\sup_{z\in H_2}\{\langle y,z\rangle-g(z)\}$. Solving the dual problem \eqref{convex12} is equivalent to solving the inclusion problem: find $v \in G$ subject to
\begin{equation}\label{convex13}
0 \in -A\partial f^*(-A^*v)-B\partial g^*(-B^*v)-c.
\end{equation}
Using similar ideas in \cite[Section 6.4]{Kim}, we adapt the 2-Step Inertial Douglas-Rachford Splitting Method in Algorithm \ref{alg5} to the following 2-Step Inertial ADMM (where we apply Algorithm \ref{alg5} is adapted to monotone inclusion problem \eqref{convex13}):

\begin{algorithm}[H]
\caption{2-Step Inertial ADMM}\label{alg6}
\begin{algorithmic}[1]
\State  Choose parameters $\delta$ and $\theta$ satisfying condition \eqref{osi}. Choose
$f\in \mathcal{F}(H_1), g\in \mathcal{F}(H_2), A \in \mathcal{B}(H_1,G), B \in \mathcal{B}(H_2,G), x_0 \in H_1, z_0 \in H_2,
\hat{v}_0 \in G, \lambda >0$. Set $n=0.$
\State  Compute as follows:
        \begin{eqnarray}
\left\{  \begin{array}{ll}
      &x_{n+1}={\rm arg}\min\limits_{x \in H_1}\Big\{f(x)+ \langle \hat{v}_n,Ax+Bz_n-c\rangle+ \frac{\lambda}{2}\|Ax+Bz_n-c\|^2 \Big\},\\
      & \hat{\eta}_n= \hat{v}_n, n=0,1,\\
      &  \hat{\eta}_n=\hat{v}_n+\theta(\hat{v}_n-\hat{v}_{n-1}+\lambda A(x_{n+1}-x_n))+\delta(\hat{v}_{n-1}-\hat{v}_{n-2}+\lambda A(x_n-x_{n-1}))\\
      & n=2,3,\ldots \\
      & z_{n+1}={\rm arg}\min\limits_{z \in H_2}\Big\{g(z)+ \langle \hat{\eta}_n,Ax_{n+1}+Bz-c\rangle+ \frac{\lambda}{2}\|Ax_{n+1}+Bz-c\|^2 \Big\},\\
      & \hat{v}_{n+1}=\hat{\eta}_n+\lambda(Ax_{n+1}+Bz_{n+1}-c)
      \end{array}
      \right.
      \end{eqnarray}
\State Set $ n \leftarrow n+1 $, and {\bf go to 2}.
\end{algorithmic}
\end{algorithm}

\noindent
 We remark that the idea of Algorithm \ref{alg6} can be applied to stochastic ADMM or its multi-block extensions in the recent works in \cite{Bai2,Bai3} for solving problems arising in machine learning because the basic model therein is a two-block separable convex optimization problem.

\section{Numerical Illustrations}\label{sec:numerics}
\noindent The focus of this section is to provide some computational experiments to demonstrate the effectiveness, accuracy and easy-to-implement nature of our proposed algorithms. We further compare our proposed schemes with some existing methods in the literature. All codes were written in MATLAB R2020b and performed on a PC Desktop
Intel(R) Core(TM) i7-6600U CPU @ 3.00GHz 3.00 GHz, RAM 32.00 GB.
\vskip 2mm

\noindent Numerical comparisons are made with the algorithm proposed in \cite{Chen}, denoted as Chen et al. Alg. 1, algorithm proposed in \cite{Kim}, denoted as Kim Alg, and algorithm proposed in \cite{Bai1}, denoted as P-PPA Alg.

\begin{exm}\label{ex1}
First, we compare our proposed 2-Step Inertial Proximal Method of Multipliers in Algorithm \ref{alg3} with the methods in \cite{Chen,Kim} using the following basis pursuit problem:
\begin{eqnarray}\label{h1}
\left\{  \begin{array}{ll}
      & \min\limits_{x \in \mathbb{R}^N} \|u\|_1\\
      & {\rm subject~ to}~~ Au=b,
      \end{array}
      \right.
\end{eqnarray}
where $A\in \mathbb{R}^{M \times N}$ and $b\in \mathbb{R}^M$. In our experiment, we consider $N= 200, 500$ and $M = 50, 100$
and $A$ is randomly generated. A true sparse $u^*$ is randomly generated followed by a
thresholding to sparsify nonzero elements, and $b$ is then given by $Au^*$. We run 100
iterations of the proximal method of multipliers and its variants with initial $x_0= \textbf{0}$. Since the $u_{n+1}$-update does not have a closed form, we used a sufficient
number of iterations to solve the $u_{n+1}$-update using the strongly convex version of
FISTA \cite{Beck} in \cite[Theorem 4.10]{Chambolle}. The stopping criterion for this example is $D_n:=||y_n - x_{n+1}||_2\leq \epsilon$, where $\epsilon=10^{-4}$. The parameter used are provided in Table \ref{table1a}

\begin{table}[H]
\caption{Methods Parameters for Example \ref{ex1}}
\centering
\begin{tabular}{c c c c c c c c c c c c c c}
 \toprule[1.5pt] \\
Proposed Alg. \ref{alg3} & $\displaystyle \lambda_1=10^{-4}$ & $\displaystyle \theta = 0.1$ & $\displaystyle \delta = -0.14412$\\
\toprule[1.5pt]\\
Chen \it{et al.} Alg. & $\displaystyle \lambda_1=10^{-4}$ & $\displaystyle \theta = 0.1$ & $\displaystyle \delta = 0$\\
\toprule[1.5pt] \\
Kim Alg. & $\displaystyle \lambda_1=10^{-4}$ \\
\toprule[1.5pt]
\end{tabular}
\label{table1a}
\end{table}

\begin{table}[H]
\caption{Example \ref{ex1} comparison for different $N$ and $M$}
\centering 
\begin{tabular}{c c c c c c c c c c c c c c}
\toprule[1.5pt]
    & \multicolumn{2}{c}{$(N,M)=(200,50)$} & \multicolumn{2}{c}{$(N,M)=(200,100)$} \\
         \cline{2-3}\cline{4-5}\cline{6-7}\cline{8-9}\\
  & No. of Iter. & CPU Time & No. of Iter. & CPU Time\\
 \toprule[1.5pt] \\
Proposed Alg. \ref{alg3}  & 3 & $4.3561\times 10^{-3}$ & 3 & $4.1955\times 10^{-3}$ \\ [0.5ex]
\hline \\
Chen \it{et al.} Alg. 1 & 3 & $5.8181\times 10^{-3}$ & 3 & $4.7572\times 10^{-3}$ \\ [0.5ex]
\hline \\
Kim Alg.  &  3 & $5.0265\times 10^{-3}$ & 3 & $5.5091\times 10^{-3}$ \\ [0.5ex]
\toprule[1.5pt]
    & \multicolumn{2}{c}{$(N,M)=(500,50)$} & \multicolumn{2}{c}{$(N,M)=(500,100)$} \\
         \cline{2-3}\cline{4-5}\cline{6-7}\cline{8-9}\\
  & No. of Iter. & CPU Time & No. of Iter. & CPU Time \\
 \toprule[1.5pt] \\
Proposed Alg. \ref{alg3}  & 3 & $2.7343\times 10^{-2}$ & 3 & $3.0797\times 10^{-2}$ \\ [0.5ex]
\hline \\
Chen \it{et al.} Alg. 1 & 3 & $3.0817\times 10^{-2}$ & 3 & $3.1057\times 10^{-2}$ \\ [0.5ex]
\hline \\
Kim Alg.  &  3 & $2.8120\times 10^{-2}$ & 3 & $3.8054\times 10^{-2}$ \\ [0.5ex]
 \hline
\end{tabular}
\label{table1b}
\end{table}

\begin{figure}[H]
\minipage{0.53\textwidth}
\includegraphics[width=\linewidth]{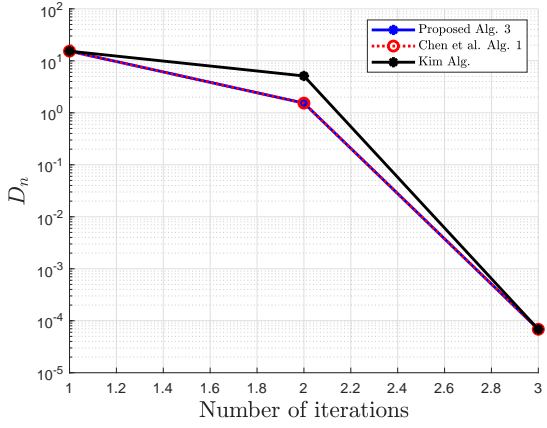}
\caption{Example \ref{ex1}: $(N,M)=(200,50)$}\label{Fig1a}
\endminipage\hfill
\minipage{0.53\textwidth}
\includegraphics[width=\linewidth]{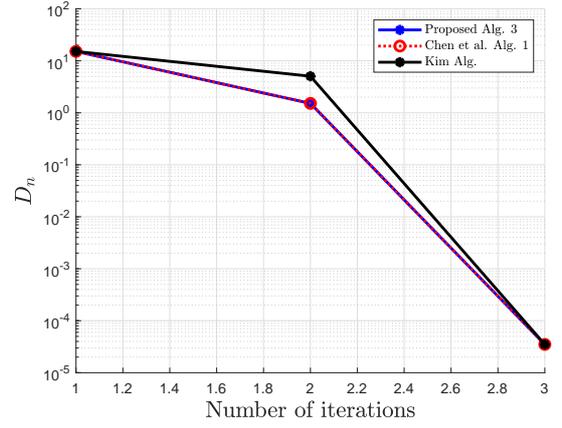}
\caption{Example \ref{ex1}: $(N,M)=(200,100)$}\label{Fig1b}
\endminipage
\end{figure}

\begin{figure}[H]
\minipage{0.53\textwidth}
\includegraphics[width=\linewidth]{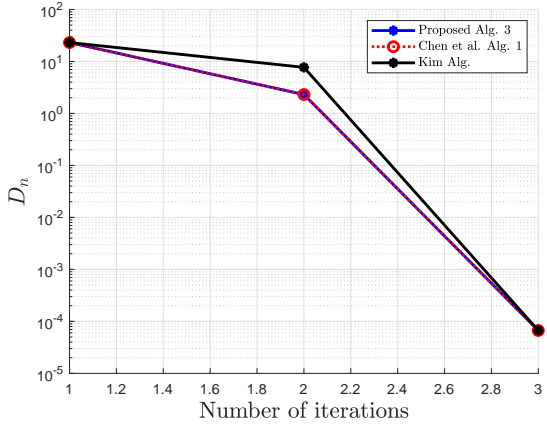}
\caption{Example \ref{ex1}: $(N,M)=(500,50)$}\label{Fig1c}
\endminipage\hfill
\minipage{0.53\textwidth}
\includegraphics[width=\linewidth]{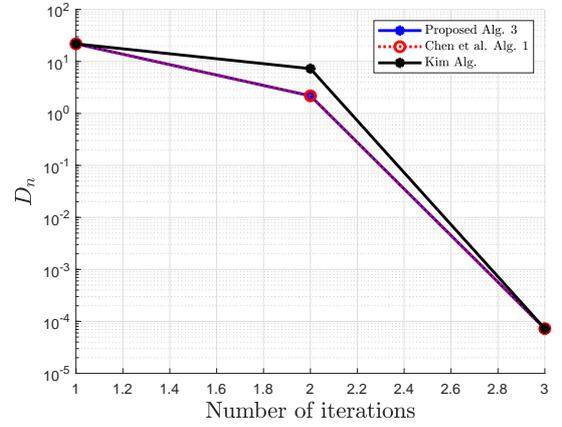}
\caption{Example \ref{ex1}: $(N,M)=(500,100)$}\label{Fig1d}
\endminipage
\end{figure}
\end{exm}

\begin{exm}\label{ex3}
In this example, we apply our proposed 2-Step Inertial ADMM (Algorithm \ref{alg6}) to the problem
\begin{eqnarray}\label{h4}
\left\{  \begin{array}{ll}
      & \min\limits_{x\in \mathbb{R}^N}\max\limits_{z\in \mathbb{R}^M} \frac{1}{2}\|Fx-b\|^2+\gamma \|z\|_1\\
      & s.t. Dx - z=0
      \end{array}
      \right.
\end{eqnarray}
where $\gamma >0$ and compare with the methods in \cite{Bai1,Chen,Kim}. The problem \eqref{h4} is associated with the total-variation-regularized least squares
problem

\begin{equation}\label{h5}
\min_{x\in \mathbb{R}^N}\frac{1}{2}\|Fx-b\|^2+\gamma \|Dx\|_1,
\end{equation}
where $F \in \mathbb{R}^{p \times N}, b\in \mathbb{R}^p$, and a matrix $D\in \mathbb{R}^{M \times N}$ is given as

\begin{eqnarray*}
D=
\left(
  \begin{array}{ccccccc}
   -1 & 1 & 0 & 0 & \ldots & 0&  \\
    0 & 1 & -1 & 0 & \ldots & 0& \\
  \vdots &\ddots  &\ddots & \ddots &\ddots&\vdots&  \\
    \vdots &\ddots  &0  & 1 & -1&0& \\
    0&\ldots &\ldots&0&1&-1&
  \end{array}
\right)
\end{eqnarray*}
Take $f(x):=\frac{1}{2}\|Fx-b\|^2, g(z):= \gamma \|z\|_1, A=D, B=-I$ and $c=0$ in \eqref{convex11}, our
2-Step Inertial ADMM \ref{alg6} reduces to:

\begin{algorithm}[H]
\caption{2-Step Inertial ADMM}\label{alg7}
\begin{algorithmic}[1]
\State  Choose parameters $\delta$ and $\theta$ satisfying condition \eqref{osi}. Choose
$x_0 \in \mathbb{R}^N, z_0 \in \mathbb{R}^M,
\hat{v}_0 \in \mathbb{R}^M, \lambda >0$. Set $n=0.$
\State  Compute as follows:
        \begin{eqnarray}
\left\{  \begin{array}{ll}
      &x_{n+1}={\rm arg}\min\limits_{x \in \mathbb{R}^N}\Big\{\frac{1}{2}\|Fx-b\|^2+ \langle \hat{v}_n,Dx-z_n\rangle+ \frac{\lambda}{2}\|Dx-z_n\|^2 \Big\}\\
      &=(F^T F+\lambda D^T D)^{-1}(D^T(\lambda z_n-\hat{v}_n)+F^T b),\\
      & \hat{\eta}_n= \hat{v}_n, n=0,1,\\
      &  \hat{\eta}_n=\hat{v}_n+\theta(\hat{v}_n-\hat{v}_{n-1}+\lambda D(x_{n+1}-x_n))+\delta(\hat{v}_{n-1}-\hat{v}_{n-2}+\lambda D(x_n-x_{n-1}))\\
      & n=2,3,\ldots \\
      & z_{n+1}={\rm arg}\min\limits_{z \in \mathbb{R}^M}\Big\{\gamma \|z\|_1+ \langle \hat{\eta}_n,Dx_{n+1}-z\rangle+ \frac{\lambda}{2}\|Dx_{n+1}-z\|^2 \Big\}\\
      & =S_{\frac{\gamma}{\lambda}}\Big(D x_{n+1}+\frac{1}{\lambda}\hat{\eta}_n \Big)\\
      & \hat{v}_{n+1}=\hat{\eta}_n+\lambda(Dx_{n+1}-z_{n+1}),
      \end{array}
      \right.
      \end{eqnarray}
\State Set $ n \leftarrow n+1 $, and {\bf go to 2}.
\end{algorithmic}
\end{algorithm}
\noindent where $S_{\tau}(z):= \max\{|z|-\tau,0\} \odot {\rm sign}(z)$ is the soft-thresholding operator,
with the element-wise absolute value, maximum and multiplication operators, $|.|,
\max\{.,.\}$ and $\odot$, respectively.\\

\noindent
In this numerical test, we consider different cases for the choices of $N$, $M$, and $p$. A true vector $x^*$
is constructed such that a vector $Dx^*$ has few nonzero elements. A matrix $F$ is randomly generated and a noisy vector $b$ is generated by adding randomly generated (noise) vector to $Fx^*$.\\

\noindent Case 1: $N=100$, $M=99$, and $p=5$\\
Case 2: $N=200$, $M=199$, and $p=10$\\
Case 3: $N=300$, $M=299$, and $p=20$\\
Case 4: $N=400$, $M=399$, and $p=40$\\

\noindent The stopping criterion for this example is $D_n\leq \epsilon$, where $\epsilon=10^{-5}$. The parameter used are provided in Table \ref{table2a}.

\begin{table}[H]
\caption{Methods Parameters for Example \ref{ex3}}
\centering
\begin{tabular}{c c c c c c c c c c c c c c}
 \toprule[1.5pt] \\
Proposed Alg. \ref{alg3} & $\displaystyle \lambda_1=0.1$ & $\displaystyle \theta = 0.1$ & $\displaystyle \delta = -10^{-3}$ &  $\displaystyle \gamma = 0.01$\\
\\
& $D_n:=||Dx_n - z_n||_2^2$\\
\toprule[1.5pt]\\
Chen \it{et al.} Alg. & $\displaystyle \lambda_1=0.1$ & $\displaystyle \theta = 0.1$ & $\displaystyle \delta = 0$ & $\displaystyle \gamma = 0.01$\\
\\
& $D_n:=||Dx_n - z_n||_2^2$\\
\toprule[1.5pt] \\
Kim Alg. & $\displaystyle \lambda_1=0.1$ & $\displaystyle \gamma = 0.01$ & $D_n:=||Dx_n - z_n||_2^2$ \\
\toprule[1.5pt] \\
 P-PPA Alg. &  $\displaystyle s = 10$ &  $\displaystyle \rho = 6$ &  $\displaystyle \varepsilon = 1.5$ &  $\displaystyle \tau = 1.5$ \\
\\
 &  $\displaystyle \sigma = \frac{1.1}{s}$ &  $\displaystyle \mu = 0.1 ||F^T b||_\infty $  & $D_n:=||y_{n+1} - y_n||_2$\\
\toprule[1.5pt]
\end{tabular}
\label{table2a}
\end{table}

\begin{table}[H]
\caption{Example \ref{ex1} comparison for different cases}
\centering 
\begin{tabular}{c c c c c c c c c c c c c c}
\toprule[1.5pt]
    & \multicolumn{2}{c}{Case 1} & \multicolumn{2}{c}{Case 2} \\
         \cline{2-3}\cline{4-5}\cline{6-7}\cline{8-9}\\
  & No. of Iter. & CPU Time & No. of Iter. & CPU Time\\
 \toprule[1.5pt] \\
Proposed Alg. \ref{alg6}  & 7 & $1.9990\times 10^{-4}$ & 21 & $1.2536\times 10^{-3}$ \\ [0.5ex]
\hline \\
Chen \it{et al.} Alg. 1 & 8 & $1.9240\times 10^{-4}$ & 43 & $2.6423\times 10^{-3}$ \\ [0.5ex]
\hline \\
Kim Alg.  &  11 & $2.9860\times 10^{-4}$ & 41 & $2.3246\times 10^{-3}$ \\ [0.5ex]
\hline \\
 P-PPA Alg.  &   39 &  $1.2939\times 10^{-3}$ &  40 &  $2.6799\times 10^{-3}$ \\ [0.5ex]
\toprule[1.5pt]
    & \multicolumn{2}{c}{Case 3} & \multicolumn{2}{c}{Case 4} \\
         \cline{2-3}\cline{4-5}\cline{6-7}\cline{8-9}\\
  & No. of Iter. & CPU Time & No. of Iter. & CPU Time \\
 \toprule[1.5pt] \\
Proposed Alg. \ref{alg3}  & 21 & $2.2213\times 10^{-3}$ & 24 & $7.2828\times 10^{-3}$ \\ [0.5ex]
\hline \\
Chen \it{et al.} Alg. 1 & 52 & $6.4371\times 10^{-3}$ & 60 & $2.3097\times 10^{-2}$ \\ [0.5ex]
\hline \\
Kim Alg.  &  24 & $2.7002\times 10^{-3}$ & 43 & $1.5199\times 10^{-2}$ \\ [0.5ex]
\hline \\
 P-PPA Alg.  &   41 &  $3.5834\times 10^{-3}$ &  44 &  $1.2796\times 10^{-2}$ \\ [0.5ex]
 \hline
\end{tabular}
\label{table2b}
\end{table}

\begin{figure}[H]
\minipage{0.53\textwidth}
\includegraphics[width=\linewidth]{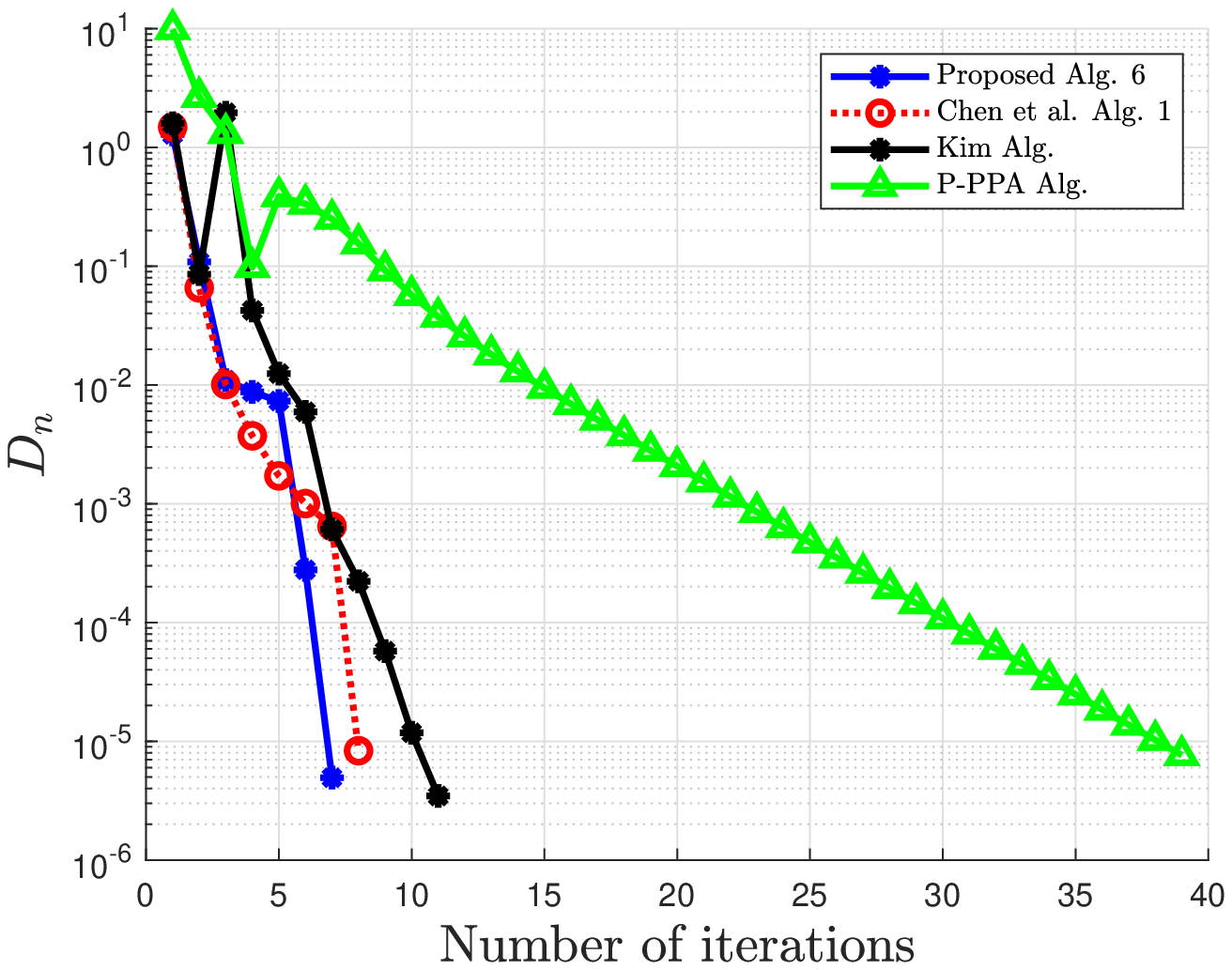}
\caption{Example \ref{ex3}: Case 1}\label{Fig2a}
\endminipage\hfill
\minipage{0.53\textwidth}
\includegraphics[width=\linewidth]{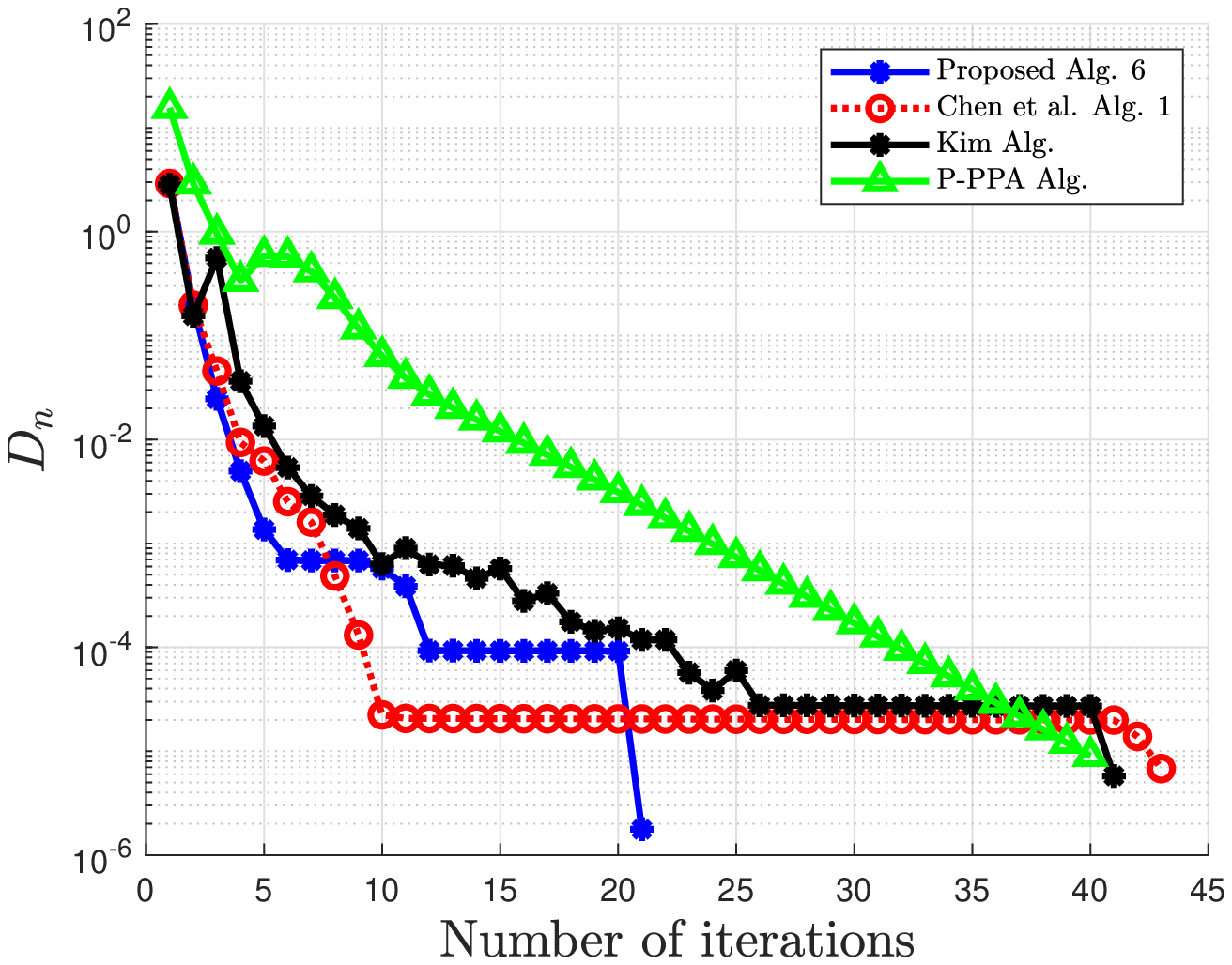}
\caption{Example \ref{ex3}: Case 2}\label{Fig2b}
\endminipage
\end{figure}

\begin{figure}[H]
\minipage{0.53\textwidth}
\includegraphics[width=\linewidth]{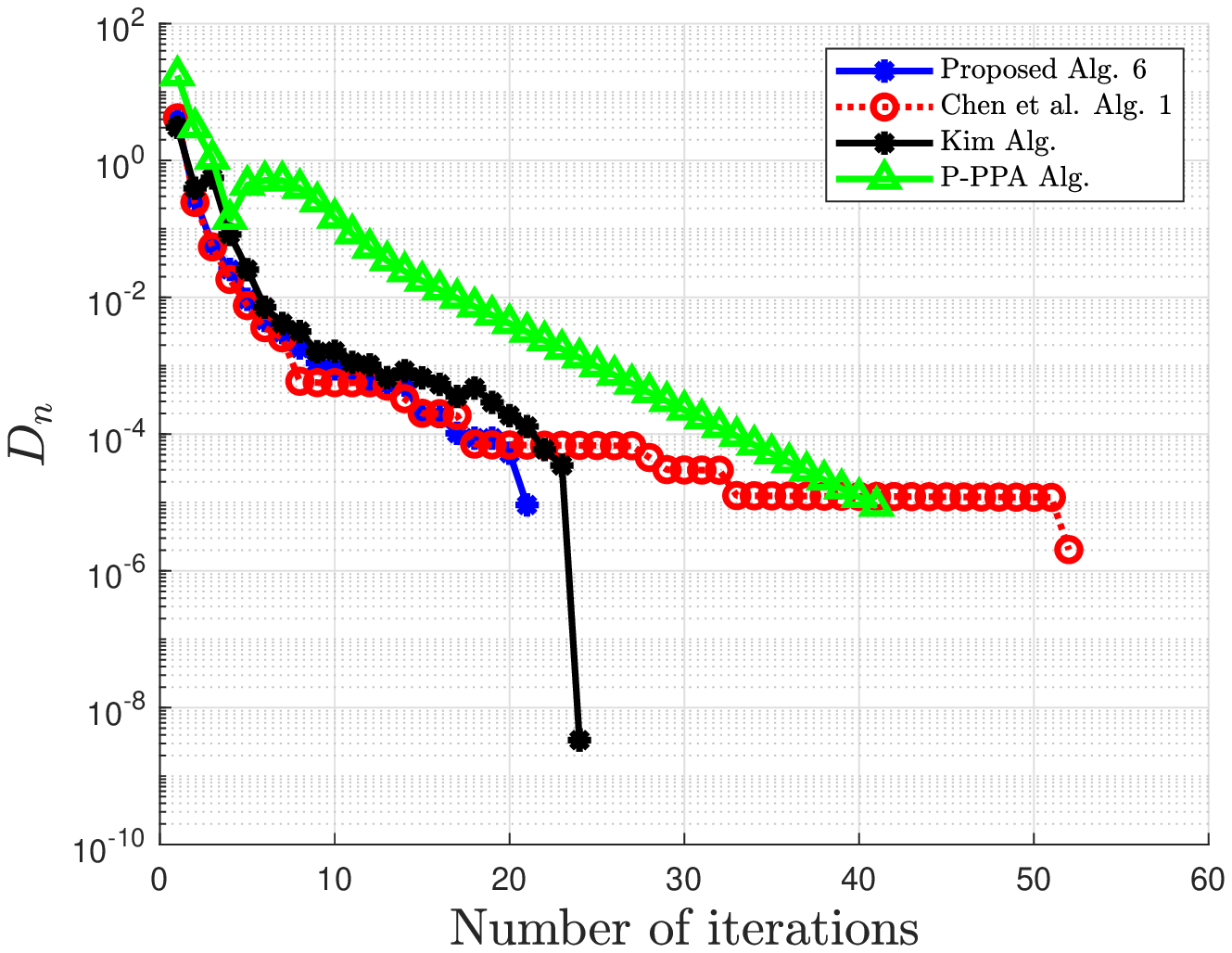}
\caption{Example \ref{ex3}: Case 3}\label{Fig2c}
\endminipage\hfill
\minipage{0.53\textwidth}
\includegraphics[width=\linewidth]{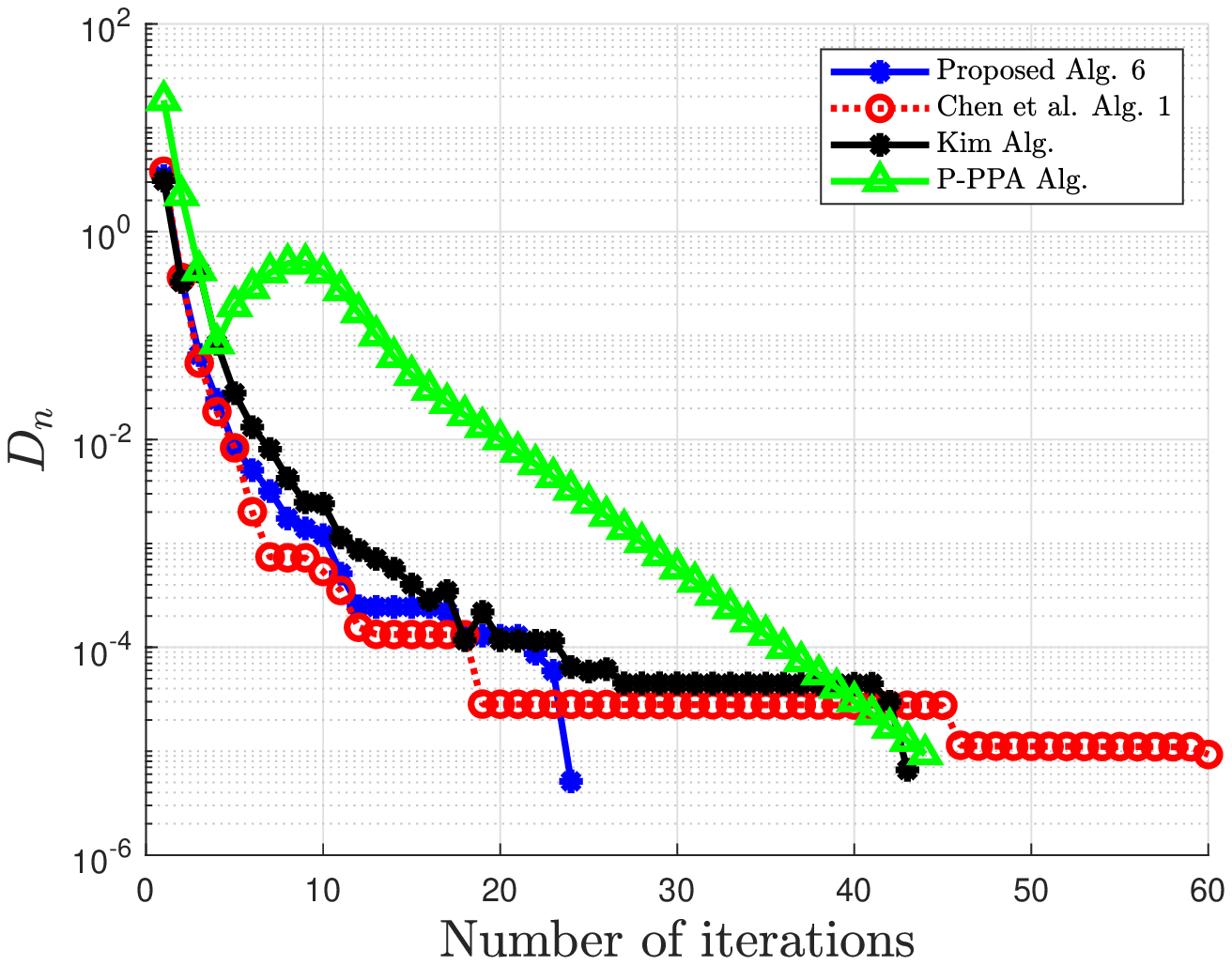}
\caption{Example \ref{ex3}: Case 4}\label{Fig2d}
\endminipage
\end{figure}

\end{exm}

\begin{rem}$\,$\\
From the above numerical Examples \ref{ex1} and \ref{ex3}, we give the following remarks.
\begin{itemize}
\item[{\rm (1).}] Noting from Figures \ref{Fig1a} - \ref{Fig2d} and Tables \ref{table1b} and \ref{table2b}, our proposed Algorithms are clearly easy to implement, efficiency and accurate.
\item[{\rm (2).}] Clearly from the numerical Example \ref{ex3}, our proposed Algorithm \ref{alg6} outperforms the algorithm proposed by Chen et al. in \cite{Chen}, the algorithm proposed by Kim in \cite{Kim},  and the algorithm proposed by Bai et al. in \cite{Bai1} with respect to the number of iterations and the CPU time. The proposed algorithm \ref{alg3} also competes favorably with Chen et al. Alg. 1 in \cite{Chen} and that of Kim in \cite{Kim} for the case of Example \ref{ex1}.
\end{itemize}

\end{rem}

\section{Conclusion}\label{conclude}
\noindent
We have introduced in this paper, a two-step inertial proximal point algorithm for monotone inclusion problem in Hilbert spaces. Weak convergence of the sequence of iterates are obtained under standard conditions and non-asymptotic $\textit{O}(1/n)$ rate of convergence in the ergodic sense given. We support the theoretical analysis of our method with numerical illustrations derived from basis pursuit problem and numerical implementation of two-step inertial ADMM. Preliminary numerical results show that our proposed method is competitive and outperforms some related and recent proximal point algorithms in the literature. Part of our future projects is to study two-step inertial proximal point algorithm with corrected term for the monotone inclusion considered in this paper.

\section*{Acknowledgements}
The authors are grateful to the associate editor and the two anonymous referees for their insightful comments and suggestions which have improved greatly on the earlier version of the paper.

\end{document}